\documentclass[ 12pt,a4paper,openany]{article}

\usepackage{amsmath,amssymb,amsfonts,amsthm}
\usepackage{graphics}
\usepackage{mathrsfs}
\usepackage[all]{xy}

\newtheorem{thm}{Theorem}[section]
\newtheorem{prop}[thm]{Proposition}
\newtheorem{lemma}[thm]{Lemma}

\theoremstyle{definition}
\newtheorem{definition}[thm]{Definition}
\newtheorem{rem}[thm]{Remark}

\newtheorem*{ack}{Acknowledgment}
\newtheorem{condition}[thm]{Condition}

\newcommand{\sq}{\hfill $\square$}
\newcommand{\kah}{K\"{a}hler}
\newcommand{\dol}{\sqrt{-1}\partial \overline{\partial}}

\makeatletter
\def\address#1#2{\begingroup
\noindent\parbox[t]{16cm}{%
\small{\scshape\ignorespaces#1}\par\vskip1ex
\noindent\small{\itshape E-mail address}%
\/: #2\par\vskip4ex}\hfill%
\endgroup}%
\makeatother

\makeatletter

\@addtoreset{equation}{section}
\makeatother

\title{Toward a construction of scalar-flat K\"{a}hler metrics on affine algebraic manifolds}
\author{Takahiro Aoi}
\date{}
\begin{document}
\maketitle


\begin{abstract}
Let $(X,L_{X})$ be an $n$-dimensional polarized manifold.\
Let $D$ be a smooth hypersurface defined by a holomorphic section of $L_{X}$.\
In this paper, we study the existence of a complete scalar-flat K\"{a}hler metric on $X \setminus D$ on the assumption that $D$ has a constant positive scalar curvature K\"{a}hler metric.\
\end{abstract}

\setcounter{tocdepth}{2}
\tableofcontents

\section{
Introduction}
\label{intro}
The existence of constant scalar curvature {\kah} (cscK) metrics on complex manifolds is a fundamental problem in {\kah} geometry.\
If a complex manifold is noncompact, there are many positive results in this problem.\
In 1979, Calabi \cite{Ca} showed that if a Fano manifold has a {\kah} Einstein metric, then there is a complete Ricci-flat {\kah} metric on the total space of the canonical line bundle.\
In addition, there exist following generalizations.\
In 1990, Bando-Kobayashi \cite{BK} showed that if a Fano manifold admits an anti-canonical smooth divisor which has a Ricci-positive {\kah} Einstein metric, then there exists a complete Ricci-flat {\kah} metric on the complement (see also \cite{TY}).\
Tian-Yau \cite{TY1} showed that if a Fano manifold admits an anti-canonical smooth divisor which has a Ricci-flat {\kah} metric, then there is a complete Ricci-flat {\kah} metric on the complement.\
In 2002, on the other hand, as a scalar curvature version of Calabi's result \cite{Ca}, Hwang-Singer \cite{HS} showed that if a polarized manifold has a nonnegative cscK metric, then the total space of the dual line bundle admits a complete scalar-flat {\kah} metric.\
However, a similar generalization of Hwang-Singer \cite{HS} like Bando-Kobayashi \cite{BK} and Tian-Yau \cite{TY1} is unknown since it is hard to solve a forth order nonlinear partial differential equation.\

In this paper, assuming the existence of a smooth hypersurface which admits a constant positive scalar curvature {\kah} metric, we study the existence of a complete scalar-flat {\kah} metric on the complement of this hypersurface.\

Let $(X,L_{X})$ be a polarized manifold of dimension $n$, i.e., $X$ is an $n$-dimensional compact complex manifold and $L_{X}$ is an ample line bundle over $X$.\
Assume that there is a smooth hypersurface $D \subset X$ with
$$
D \in |L_{X}|.
$$
Set an ample line bundle $L_{D} := \mathscr{O}(D)|_{D} = L_{X}|_{D}$ over $D$.\
Since $L_{X}$ is ample, there exists a Hermitian metric $h_{X}$ on $L_{X}$ which defines a {\kah} metric $\theta_{X}$ on $X$, i.e., the curvature form of $h_{X}$ multiplied by $\sqrt{-1}$ is positive definite.\
Then, the restriction of $h_{X}$ to $L_{D}$ defines also a {\kah} metric $\theta_{D}$ on $D$.\ 
Let $\hat{S}_{D}$ be the average of the scalar curvature $S(\theta_{D})$ of $\theta_{D}$ defined by
$$
\hat{S}_{D} := \frac{\displaystyle\int_{D} S(\theta_{D}) \theta_{D}^{n-1}}{\displaystyle\int_{D} \theta_{D}^{n-1}} = \frac{(n-1) c_{1}(K_{D}^{-1}) \cup c_{1}(L_{D})^{n-2}}{c_{1}(L_{D})^{n-1}},
$$
where $K_{D}^{-1}$ is the anti-canonical line bundle of $D$.\
Note that $\hat{S}_{D}$ is a topological invariant in the sense that it is representable in terms of Chern classes of the line bundles $K_{D}^{-1}$ and $L_{D}$.\
In this paper, we treat the following case :
\begin{equation}
\label{positivity 1}
\hat{S}_{D} > 0.
\end{equation}
Let $\sigma_{D} \in H^{0}(X,L_{X})$ be a defining section of $D$ and set $t := \log || \sigma_{D} ||_{h_{X}}^{-2}$.\
Following \cite{BK}, we can define a complete {\kah} metric $\omega_{0}$ by
$$
\omega_{0}:= \frac{n(n-1)}{\hat{S}_{D}}\dol \exp \left(\frac{\hat{S}_{D}}{n(n-1)}t\right)\\
$$
on the noncompact complex manifold $X \setminus D$.\
In addition, since $(X \setminus D, \omega_{0})$ is of asymptotically conical geometry (see \cite{BK} or Section 4 of this paper), we can define the weighted Banach space $C^{k,\alpha}_{\delta} = C^{k,\alpha}_{\delta}(X \setminus D)$ for $k \in \mathbb{Z}_{\geq 0}, \alpha \in (0,1)$ and with a weight $\delta \in \mathbb{R}$ with respect to the distance function $r$ defined by $\omega_{0}$ from some fixed point in $X \setminus D$.\
It follows from the construction of $\omega_{0}$ that $S(\omega_{0}) = O(r^{-2})$ near $D$.\

\medskip
The cscK condition implies the following stronger decay property.\

\begin{thm}
\label{scalar curvature decay}
If $\theta_{D}$ is a constant positive scalar curvature {\kah} metric on $D$, i.e., $S(\theta_{D}) = \hat{S}_{D} > 0,$
we have
$$
S(\omega_{0}) = O(r^{-2 -2n(n-1)/\hat{S}_{D}})
$$
as $r \to \infty$.\
\end{thm}
Thus, the cscK condition implies that $S(\omega_{0}) \in C^{k,\alpha}_{\delta}$ for some $\delta > 2$ and  any $k,\alpha$.\

\medskip
In order to construct a complete scalar-flat {\kah} metric on $X \setminus D$, the linearization of the scalar curvature operator plays an important role :
$$
L_{\omega_{0}} = - \mathcal{D}^{*}_{\omega_{0}}\mathcal{D}_{\omega_{0}} + (\nabla^{1,0} \ast , \nabla^{0,1} S(\omega_{0}))_{\omega_{0}}.
$$
Here, $\mathcal{D}_{\omega_{0}} = \overline{\partial} \circ \nabla^{1,0}$.\
We will show that if $4 < \delta <2n $ and there is no nonzero holomorphic vector field on $X$ which vanishes on $D$, then $\mathcal{D}^{*}_{\omega_{0}}\mathcal{D}_{\omega_{0}} : C^{4,\alpha}_{\delta - 4} \to C^{0,\alpha}_{\delta}$ is isomorphic.\
For such operators, we consider the following :

\begin{condition}
\label{condition A}
Assume that $n \geq 3$ and there is no nonzero holomorphic vector field on $X$ which vanishes on $D$.\
For $4 < \delta <2n $, the operator
$$
L_{\omega_{0}} : C^{4,\alpha}_{\delta - 4} \to C^{0,\alpha}_{\delta}
$$
is isomorphic, i.e., we can find a constant $\hat{K}>0$ such that
$$
||L_{\omega_{0}} \phi ||_{C^{0,\alpha}_{\delta}} \geq \hat{K} || \phi ||_{C^{4,\alpha}_{\delta - 4}}
$$
for any $ \phi \in C^{4,\alpha}_{\delta - 4} $.\
\end{condition}

In addition, we consider
\begin{condition}
\label{condition B}
$$
|| S(\omega_{0}) ||_{C^{0,\alpha}_{\delta}} < c_{0}\hat{K}/2.
$$
\end{condition}

Here, the constant $c_{0}$ will be defined in Lemma \ref{contraction lemma} later.\
Under these conditions, Theorem \ref{scalar curvature decay} implies the following result :

\begin{thm}
\label{complete scalar-flat}
Assume that $n \geq 3$ and there is no nonzero holomorphic vector field on $X$ which vanishes on $D$.\
Assume that $\theta_{D}$ is a constant scalar curvature {\kah} metric satisfying
$$
0 < \hat{S}_{D} < n(n-1).
$$
Assume moreover that Condition $\ref{condition A}$ and Condition $\ref{condition B}$ hold,
then $X \setminus D$ admits a complete scalar-flat {\kah} metric.\
\end{thm}

In fact, we can show the existence of a complete scalar-flat {\kah} metric on $X \setminus D$ under the following assumptions : (i) $n \geq 3$ and there is no nonzero holomorphic vector field on $X$ which vanishes on $D$, (ii) there exists a complete {\kah} metric on $X \setminus D$ which is of asymptotically conical geometry, such that its scalar curvature is sufficiently small and decays at a higher order.\
So, if there exists a complete {\kah} metric on $X \setminus D$ which is sufficiently close to $\omega_{0}$ at infinity, satisfying Condition \ref{condition A} and Condition \ref{condition B}, we can show the existence of a complete scalar-flat {\kah} metric on $X \setminus D$.\
Theorem \ref{complete scalar-flat} is proved by the fixed point theorem on the weighted Banach space $C^{4,\alpha}_{\delta - 4}(X \setminus D)$ by following Arezzo-Pacard \cite{AP1}, \cite{AP2} (see also \cite{Sz1}).\
In general, constants $c_{0}, \hat{K}$ which arise in Condition \ref{condition A} and Condition \ref{condition B} depend on the background {\kah} metric $\omega_{0}$.\
In addition, to construct such a {\kah} metric, we have to find a complete {\kah} metric $X \setminus D$ whose scalar curvature is arbitrarily small.\
We will handle this problem in next papers \cite{Aoi2} and \cite{Aoi3}.

This paper is organized as follows.\
In Section 2, recalling the result due to Hwang-Singer \cite{HS}, we give the volume growth of a geodesic ball with respect to the {\kah} metric obtained in \cite{HS}.\
This case is a toy-model of our problem.\
In Section 3, we prove Theorem \ref{scalar curvature decay}.\
To prove this, we use fundamental results in matrix analysis.\
In Section 4, we will introduce the asymptotically conicalness of open Riemannian manifolds and weighted Banach spaces by following Bando-Kobayashi \cite{BK}.\
In Section 5, we study the linearization of the scalar curvature operator between some weighted Banach spaces.\
In section 6, we prove Theorem \ref{complete scalar-flat} by following Arezzo-Pacard \cite{AP1}, \cite{AP2} (see also \cite{Sz1}).\

\begin{ack}
The author would like to thank Professor Ryoichi Kobayashi who first brought the problem in this paper to his attention, for many helpful comments.\
\end{ack}

\section{
The case of line bundles}
\label{sec:2}

Before considering the general case, we consider the existence of a complete scalar-flat {\kah} metrics on line bundles and compute the volume growth.\
Let $(X,L)$ be an $n$-dimensional polarized manifold and $\theta \in 2\pi c_{1}(L)$ be a cscK metric.\
In this case, the value of the scalar curvature is equal to the following average value given by
\begin{equation}
\hat{S}_{X} := \frac{nc_{1}(X) \cup c_{1}(L)^{n-1}}{c_{1}(L)^{n}},
\end{equation}
where $(2\pi)^{n} c_{1}(L)^{n} = \int_{X} \theta^{n}$ and $(2\pi)^{n} c_{1}(X) \cup c_{1}(L)^{n-1} = \int_{X} Ric(\theta) \wedge \theta^{n-1}$.\
Note that $\hat{S}_{X}$ is a topological invariant.\
In 2002, Hwang-Singer \cite{HS} showed the following:
\begin{thm}
\label{scalar-flat line}
Let $(X,L)$ be an $n$-dimensional polarized manifold.\
Suppose that there exists a constant scalar curvature {\kah} metric $\theta \in 2 \pi c_{1}(L)$ and the value of the scalar curvature of $\theta$ is nonnegative:\
\begin{equation}
\hat{S}_{X} \geq 0.
\label{scalar condition}
\end{equation}
Then, there exists a complete scalar-flat {\kah} metric $\omega$ on the total space of the dual line bundle $L^{-1}$.\
\end{thm}

\begin{rem}
In \cite{HS}, they treat more general cases which contain the existence of a complete scalar-flat {\kah} metric on the disc bundle in $L^{-1}$.\
In this article, it is enough for us to consider Theorem \ref{scalar-flat line}.\
\end{rem}
To compute the volume growth of the {\kah} metric $\omega$ above, we need to recall the proof of Theorem \ref{scalar-flat line} by following \cite{HS} (see also \cite{Sz1}).\

\subsection{
LeBrun-Simanca metrics}
\label{sec:2.1}
In this subsection, we construct the LeBrun-Simanca metric for an ample line bundle.\
This metric for the dual of the tautological line bundle over $\mathbb{C}\mathbb{P}^{n-1}$ was found by LeBrun and Simanca \cite{Le},\cite{Si}(see also \cite{Sz1}).\

Fix an $n$-dimensional polarized manifold $(X,L)$.\
Consider a Hermitian metric $h$ on $L$ which defines the {\kah} metric $\theta \in 2\pi c_{1}(L)$.\
Let $p:L^{-1} \to X$ be the projection map and $(L^{-1})^{*}$ be the complement of the zero section of $L^{-1}$.\
Define a smooth function $s$ by
\begin{equation}
s:(L^{-1})^{*} \rightarrow \mathbb{R},
\hspace{10pt}
\xi \mapsto \log h^{-1}(\xi,\xi),
\end{equation}
where $h^{-1}$ is a Hermitian metric on $L^{-1}$ induced by $h$.\

\begin{definition}
{\it The LeBrun-Simanca metric} $\omega$ on $(L^{-1})^{*}$ is defined by
\begin{equation}
\omega := {\dol}f(s),
\end{equation}
where $f$ is a smooth, increasing and strictly convex function on $\mathbb{R}$.
\end{definition}

To extend $\omega$ to the whole space $L^{-1}$,\
we start to compute $\omega$ in local coordinates.\
Fix a point $z_{0}$ in the base space $X$.\
Then we can find a local holomorphic coordinate chart $U$ around $z_{0}$ with a local holomorphic trivialization of $L^{-1}$ around $z_{0}$:
\begin{equation}
\label{trivialization}
(L^{-1}){\mid}_{U} \cong U \times \mathbb{C}, \hspace{5pt} \xi \mapsto (z,w),
\end{equation}
where $w$ is a fiber coodinate.\
Then, in these local coodinates, we can write as
\begin{equation}
s(z,w) = \log |w|^{2} - \log h(z),
\end{equation}
where $h(z)$ is a positive function defined by some local non-vanishing holomorphic section of $L$ on $U$.\
For any point $z_{0} \in X$,\
we can choose the above trivialization so that:
\begin{eqnarray}
d \log h(z_{0})
=0.
\end{eqnarray}
Let us compute at a point $(z_{0},w)$ with $w \neq 0$;
\begin{eqnarray}
{\dol}f(s)
&=&\ddot{f}(s) \sqrt{-1} \frac{dw \wedge d\overline{w}}{|w|^{2}} + \dot{f}(s)p^{*}\theta ,
\end{eqnarray}
where the symbol $\dot{f}$ denotes the differential of $f$ with respect to the variable $s$.\

To simplify the construction of a scalar-flat {\kah} metric on $L^{-1}$, following \cite{Sz1}, we introduce Legendre transforms and momentum profiles and recall fundamental facts of them.\
By using them,\
we can give a condition on the extension of $\omega$ to the whole space $L^{-1}$ and compute the scalar curvature not as a nonlinear PDE in forth order but as an ODE in second order by following \cite{HS}.\

\begin{definition}
Let $f$ be a strictly convex and smooth function on $\mathbb{R}$.\
Set $\tau := \dot{f}(s)$ and $I$ be an image of $\tau$.\
{\it The Legendre transform} $F$ on $I$ of $f$ with variable $\tau$ is defined by
\[
s\tau = f(s) + F(\tau).
\]
\end{definition}
Note that there are following relations:
\[
F^{\prime}(\tau) = s, \hspace{7pt}  F^{\prime\prime}(\tau) = \frac{1}{\ddot{f}(s)},
\]
where we use the symbol $F^{\prime}(\tau)$ as the differential of $F$ with respect to the variable $\tau$.\

\begin{definition}
Let $I\subset \mathbb{R}$ be an image of $\tau$.\
{\it The momentum profile} $\varphi$ of the metric $\omega = \dol f(s)$ is defined by the following:
\begin{equation}
\varphi:I\rightarrow\mathbb{R}, \hspace{10pt} \varphi(\tau)=\frac{1}{F^{\prime\prime}(\tau)} ,
\end{equation}
where $F$ is the Legendre transform of $f$ defined above.
\end{definition}
Clearly, there are following relations:
\begin{equation}
\varphi(\tau)=\ddot{f}(s), \hspace{10pt} \frac{d\tau}{ds}=\varphi(\tau) .
\end{equation}
The following proposition is the converse of the above construction:\

\begin{prop}
\label{main proposition}
Let $I \subset \mathbb{R}$ be any interval and $\varphi$ be a smooth positive function defined on $I$.\
Then we can find a smooth and strictly convex function $f$ on some interval $J$ of $\mathbb{R}$ such that 
\begin{equation}
\tau=\dot{f}(s), \hspace{10pt}
\varphi(\tau)=\ddot{f}(s).
\end{equation}

\end{prop}

\begin{proof}
Let $G = G(\tau)$ be a function on $I$ with $G^{\prime}(\tau) = 1/\varphi(\tau)$.\
Since $G$ is strictly monotone increasing, we have $\tau = G^{-1}(s)$.\
Set $J := G(I)$.\
Proposition \ref{main proposition} is proved by setting 
\[
f(s) := \int^{s}_{c} G^{-1}(t)dt
\]
for some $c \in J$.
\end{proof}

\subsection{
The extension to the total space}

In this subsection, we give a condition such that the LeBrun-Simanca metric $\omega$ can be extended to the whole space $L^{-1}$ as a {\kah} metric by following \cite{HS} (see also \cite{Sz1}).\
By using the momentum profile $\varphi$, $\omega$ can be rewritten as follows:\
\begin{equation}
\label{eta formula}
\omega = \varphi(\tau) \sqrt{-1} \frac{dw \wedge d\overline{w}}{|w|^{2}} + \tau p^{*}\theta .
\end{equation}

First, we set a momentum profile $\varphi$ defined on $I := (1,N)$ for some $N \in (1,\infty]$ so that a function $f$ is defined on $J = \mathbb{R}$ in the way in the proof of Proposition \ref{main proposition}.\
Then, the formula (\ref{eta formula}) implies that the LeBrun-Simanca metric $\omega$ is positive in the base direction at any point in the zero section.\
Following \cite{HS}, to obtain the positivity in the fiber direction and the smoothness of $\omega$ on the whole space $L^{-1}$,\
we pose the boundary condition on $\varphi$:\

\begin{prop}
\label{kahlerity}

Suppose that $\varphi$ satisfies the following boundary condition:
\begin{equation}
\varphi(1)=0, \hspace{10pt} \varphi^{\prime}(1)=1,
\end{equation}
and can be extended smoothly in a neighborhood of $1$.\
Then $\omega$ can be extended to $L^{-1}$ as a {\kah} metric.\
\end{prop}
For simplicity, we will denote the extended metric on $L^{-1}$ by the same symbol $\omega$.\

\subsection{
The Ricci form and the scalar curvature}

In this subsection, we compute the Ricci form and the scalar curvature of the LeBrun-Simanca metric $\omega$ (see \cite{HS}, \cite{Sz1}).\

\begin{prop}
\label{ricci and scalar}
The Ricci form $Ric(\omega)$ and the scalar curvature $S(\omega)$ of $\omega$ are given by
\begin{equation*}
Ric(\omega) = -\varphi \Bigl( \varphi^{\prime} + n\frac{\varphi}{\tau} \Bigr)^{\prime} \frac{\sqrt{-1}dw \wedge d\overline{w}}{|w|^{2}} - \Bigl( \varphi^{\prime} + n\frac{\varphi}{\tau} \Bigr)p^{*}\theta + p^{*}Ric(\theta),
\end{equation*}
\begin{equation*}
\label{scalar}
S(\omega) = \frac{p^{*}S(\theta)}{\tau} - \frac{1}{\tau^{n}}\frac{d^{2}}{d\tau^{2}}(\tau^{n}\varphi(\tau)),
\end{equation*}
where $Ric(\omega)$ is a pointwise formula.
\end{prop}

\begin{proof}

First, the Ricci form $Ric(\omega)$ is locally given by the following:
\begin{equation*}
Ric (\omega) = -{\dol}\log {\omega}^{n+1}.
\end{equation*}
By the direct computation  at a point $z_{0}$, we have
\begin{eqnarray*}
{\omega}^{n+1} &=& \frac{\ddot{f}(s) {\dot{f}(s)}^{n}p^{*}\theta^{n} \wedge \sqrt{-1}dw \wedge d\overline{w}}{|w|^{2}}\\
&=& \frac{\varphi(\tau) {\tau}^{n}p^{*}\theta^{n} \wedge \sqrt{-1}dw \wedge d\overline{w}}{|w|^{2}}.
\end{eqnarray*}
If we choose another trivialization $(z,\hat{w})$ of $L^{-1}$,\
there exists a holomorphic transform function $g$ such that $\hat{w} = g(z)w$.\
The differential of $g(z)$ does not affect the above formula because $p^{*}\theta^{n}$ is the top wedge product in the base direction.\
Therefore, ${\omega}^{n+1}$ is invariant under the choice of the local coodinates $(z,w)$ and we can use the formula above globally.\

Let us compute $Ric(\omega)$ at a point $(z_{0},w_{0})$:
\begin{eqnarray*}
Ric(\omega)&=&-{\dol}\log\varphi(\tau(s)) - n{\dol}\log \tau(s) + p^{*}Ric(\theta)\nonumber\\
&=&-\varphi \Bigl( \varphi^{\prime} + n\frac{\varphi}{\tau} \Bigr)^{\prime} \frac{\sqrt{-1}dw \wedge d\overline{w}}{|w|^{2}} - \Bigl( \varphi^{\prime} + n\frac{\varphi}{\tau} \Bigr)p^{*}\theta + p^{*}Ric(\theta).
\end{eqnarray*}
Note that the equation of $Ric(\omega)$ is completely divided into the base direction and fiber direction.\
Taking a trace of the Ricci form by the metric $\omega = \varphi(\tau)\sqrt{-1} dw \wedge d\overline{w}/|w|^{2} + \tau p^{*}\theta$,\
we have the following:\
\begin{eqnarray*}
S(\omega) &=& -\Bigl( \varphi^{\prime} + n\frac{\varphi}{\tau} \Bigr)^{\prime} - \frac{n}{\tau}\Bigl( \varphi^{\prime} + n\frac{\varphi}{\tau} \Bigr) + \frac{p^{*}S(\theta)}{\tau}\\
&=& - \frac{1}{\tau^{n}}\frac{d^{2}}{d\tau^{2}}(\tau^{n}\varphi(\tau)) + \frac{p^{*}S(\theta)}{\tau}.
\end{eqnarray*}
Thus, the proof of Proposition \ref{ricci and scalar} is finished.
\end{proof}

\subsection{
ODE}

In this subsection, we prove Theorem \ref{scalar-flat line} by using Proposition \ref{ricci and scalar}.\
The key of the proof is that we can consider the scalar-flat condition as the case of ordinary differential equations (ODE) in second order on the assumption that $\omega$ is cscK.\\

\noindent
{\it  Proof of Theorem \ref{scalar-flat line}.}\
If the {\kah} metric $\theta$ has a constant scalar curvature,\
the value of the scalar curvature $S(\theta)$ is equal to the average of the scalar curvature:
\begin{equation}
\hat{S}_{X}=\frac{nc_{1}(X) \cup c_{1}(L)^{n-1}}{c_{1}(L)^{n}}.
\end{equation}
By the formula in Proposition \ref{ricci and scalar}, to make $\omega$ scalar-flat,\
it is enough to solve the following ODE with the boundary condition:
\begin{equation*}
\frac{d^{2}}{d\tau^{2}}(\tau^{n}\varphi(\tau)) = \hat{S}_{X}\tau^{n-1}, \hspace{10pt} \varphi(1)=0, \hspace{10pt} \varphi^{\prime}(1)=1.
\end{equation*} 
In fact, a solution of this is easily given by
\begin{equation}
\label{solution 2}
\varphi(\tau) = \frac{\hat{S}_{X}}{n(n+1)}\tau - \Bigl( \frac{\hat{S}_{X}}{n} - 1 \Bigr)\tau^{1-n} + \Bigl( \frac{\hat{S}_{X}}{n + 1} - 1 \Bigr)\tau^{-n}.
\end{equation}
If $ N = \infty $, $\varphi = O(\tau)$ as $\tau \to \infty$.\
If $ N < \infty $, $\varphi$ vanishes like a polynomial.\
Recall that $s = \int \varphi^{-1} d\tau$.\
In both cases, applying Proposition \ref{main proposition} for these $\varphi$, we can obtain an increasing, strictly convex and smooth function $f$ defined on $\mathbb{R}$ by setting an interval $I := \{ \tau \in \mathbb{R}$ $|$ $ \varphi(\tau)$ is positive $\}$.\
Then, we have finished the proof of Theorem \ref{scalar-flat line}.\sq

\subsection{
Volume growth}

In this subsection, we compute the volume growth of the {\kah} metric $\omega$.\

\begin{prop}
\label{volume growth}
Fix a point $\xi \in L^{-1}$.\
Suppose that $\omega$ be the LeBrun-Simanca metric as above.\
Let $B(\xi,r)$ be a geodesic ball with respect to $\omega$ in $L^{-1}$ of radius $r$ centered at $\xi$.\

$(1)$ If $\hat{S}_{X} >0$, we have
\begin{equation}
\int_{B(\xi,r)} \omega^{n+1} = O (r^{2(n+1)}) \hspace{7pt} \mbox{as} \hspace{7pt} r \to \infty.
\end{equation}

$(2)$ If $\hat{S}_{X} = 0$, we have
\begin{equation}
\int_{B(\xi,r)} \omega^{n+1} = O (r^{2}) \hspace{7pt} \mbox{as} \hspace{7pt} r \to \infty.
\end{equation}
\end{prop}

For a point $y \in X$, a symbol $\xi_{y}$ denotes an element of the fiber $L^{-1}_{y}$.\
In particular, a symbol $0_{y}$ denotes the zero element of the fiber $L^{-1}_{y}$.\
In this section, we use the LeBrun-Simanca metric $\omega$ given by the solution (\ref{solution 2}) in the both cases of $\hat{S}_{X}>0$ and $\hat{S}_{X}=0$.\
First, we compute a relation between the geodesic distance for the metric $\omega$ and the Hermitian norm for $h^{-1}$.\
For simplicity, $|\xi_{y}|$ denotes a square root of the hermitian norm $h^{-1}(\xi_{y},\xi_{y})$.\

\begin{lemma}\
\label{app lemma}

$(a)$ If $\hat{S}_{X}>0$, we have
\begin{equation}
d(0_{y},\xi_{y}) = O(|\xi_{y}|^{\frac{\hat{S}_{X}}{n(n+1)}}) \hspace{7pt} {\mbox as} \hspace{7pt} |\xi_{y}| \to \infty
\end{equation}
and
\begin{equation}
\tau = O(d(0_{y},\xi_{y})^{2}) \hspace{7pt} {\mbox as} \hspace{7pt} |\xi_{y}| \to \infty.
\label{tau order1}
\end{equation}

$(b)$ If $\hat{S}_{X}=0$, we have
\begin{equation}
d(0_{y},\xi_{y}) = O((\log |\xi_{y}|)^{\frac{n+1}{2n}}) \hspace{7pt} {\mbox as} \hspace{7pt} |\xi_{y}| \to \infty
\end{equation}
and
\begin{equation}
\tau = O( d(0_{y},\xi_{y})^{\frac{2}{n+1}}) \hspace{7pt} {\mbox as} \hspace{7pt} |\xi_{y}| \to \infty.
\label{tau order2}
\end{equation}

\end{lemma}

\begin{proof}

First, we prove the statement $(a)$.\
By the completeness of $\omega$, there exists a length minimizing geodesic connecting any pair of two points in $L^{-1}$.\
Fix the length minimizing geodesic $\gamma(t), t \in [0,1]$ from $0_{y}$ to $\xi_{y}$.\
Clearly, for fixed $y \in X$, the image of the geodesic $\gamma(t)$ is in the fiber $L^{-1}_{y}$.\
For simplicity, we assume that $h^{-1}(\xi_{y},\xi_{y}) = |w|^{2}$ in the trivialization (\ref{trivialization}).\
Set $v=|w|$.\

$(1)$ Recall that the positive function $\varphi(\tau)$ in the case $(a)$ is written as
\[
\varphi(\tau) = \frac{\hat{S}_{X}}{n(n+1)}\tau - \Bigl( \frac{\hat{S}_{X}}{n} - 1 \Bigr)\tau^{1-n} + \Bigl( \frac{\hat{S}_{X}}{n + 1} - 1 \Bigr)\tau^{-n}.
\]
Since $\tau \to \infty$ as $v \to \infty$,\
the second and third terms above are very small as $v \to \infty$.\
Then,
\begin{eqnarray}
\frac{d\tau}{dv}
&=& \varphi(\tau) \frac{2}{v} \nonumber\\
&=& \frac{2\hat{S}}{n(n+1)}\frac{\tau}{v} + \mbox{l.o.t}.
\end{eqnarray}
Here the symbol $\mbox{l.o.t}.$ denotes lower order terms as $v \to \infty$.\
We have
\begin{eqnarray*}
\frac{1}{\tau}\frac{d\tau}{dv} = \frac{2\hat{S}}{n(n+1)}\frac{1}{v} + \mbox{l.o.t}.
\end{eqnarray*}
Thus, we have
\begin{eqnarray}
\tau 
&=& O(|\xi_{y}|^{\frac{2\hat{S}_{X}}{n(n+1)}}),
\label{case 1.1}
\end{eqnarray}
as $|\xi_{y}| \to \infty$.\

Denote $\gamma(t) = (y,w_{t})$.\
Since the LeBrun-Simanca metric $\omega$ is $S^{1}$-invariant,\
we can write
\[
w_{t} = \theta_{t} w_{1},
\]
where $\theta_{t}$ is some real nonnegative function such that $\theta_{0}=0$ and $\theta_{1}=1$.\
Since it is enough to compute for sufficiently large $|\xi_{y}|$, we have

\begin{eqnarray}
d(0_{y},\xi_{y})
&=& \int_{0}^{1} \sqrt{\omega (\dot{\gamma(t)},\dot{\gamma(t)})}dt \nonumber\\
&=& \int_{0}^{1} \sqrt{\frac{|\dot{w_{t}}|^{2}}{|w_{t}|^{2}} \varphi(\tau)}dt \nonumber\\
&=& \int_{0}^{1} \frac{\dot{\theta_{t}}}{\theta_{t}} \sqrt{\varphi(\tau)}dt \nonumber\\
&=& \int_{c}^{1} |w_{1}|^{\frac{\hat{S}}{n(n+1)}} \dot{\theta_{t}} \theta_{t}^{\frac{\hat{S}}{n(n+1)} - 1}dt +  \mbox{l.o.t}. \nonumber\\
&=& O( |\xi_{y}|^{\frac{\hat{S}}{n(n+1)}})
\label{case 1.2}
\end{eqnarray}
as $|\xi_{y}| \to \infty$,\
where $c \in (0,1)$ is some fixed constant.\
Thus, the statement $(a)$ follows.\

$(b)$ Recall that the positive function $\varphi(\tau)$ in the case $(b)$ is written as
\[
\varphi(\tau) = \tau^{1-n} - \tau^{-n}.
\]
Similarly, we have
\begin{eqnarray*}
\frac{d\tau}{dv}
&=& \frac{2\tau^{1-n}}{v} + \mbox{l.o.t.}
\end{eqnarray*}
Then,
\begin{eqnarray*}
\tau^{n} = O(\log |\xi_{y}|),
\label{case 2.1}
\end{eqnarray*}
as $|\xi_{y}| \to \infty$.\

Denote $\gamma(t) = (y,w_{t})$.\
Similarly, we have
\begin{eqnarray}
d(0_{y},\xi_{y})
&=& \int_{0}^{1} \sqrt{\omega (\dot{\xi_{y}^{t}},\dot{\xi_{y}^{t}})}dt \nonumber\\
&=& \int_{c}^{1} \frac{\dot{\theta_{t}}}{\theta_{t}} \tau^{\frac{1-n}{2}} dt + \mbox{l.o.t}. \nonumber\\
&=& A \int_{c}^{1} \frac{\dot{\theta_{t}}}{\theta_{t}} (\log |w_{1}| + \log \theta_{t})^{\frac{1-n}{2n}} dt + \mbox{l.o.t}. \nonumber\\
&=& O((\log |\xi_{y}|)^{\frac{n+1}{2n}})
\label{case 2.2}
\end{eqnarray}
as $|\xi_{y}| \to \infty$,\
where $c \in (0,1)$ and $A>0$ are some fixed constants.\
Thus, the statement $(b)$ follows.
\end{proof}

Using Lemma \ref{app lemma}, we prove Proposition \ref{volume growth}.\

{\it Proof of Proposition $\ref{volume growth}$.}\
It is enough to compute the volume growth of a geodesic ball in $L^{-1}$ of radius $r$ centered at $0_{x}$.\
Since the restriction of $\omega$ to the zero section is $\theta$, we have
\begin{eqnarray*}
d(0_{x},\xi_{y})
&=& d(0_{x},0_{y}) + d(0_{y},\xi_{y}) \nonumber\\
&\leq& \mbox{diam}(X,\theta) + d(0_{y},\xi_{y}).
\end{eqnarray*}
Thus,
\begin{equation}
- \mbox{diam}(X,\theta) + d(0_{y},\xi_{y}) \leq d(0_{x},\xi_{y}) \leq \mbox{diam}(X,\theta) + d(0_{y},\xi_{y}).
\label{case 3}
\end{equation}

By Stokes' theorem,  we have
\begin{eqnarray}
\int_{B(0_{x},r)} \omega^{n+1}
&=& \int_{B(0_{x},r)} \tau^{n} \varphi(\tau) \sqrt{-1} \partial s \wedge \overline{\partial} s \wedge p^{*} \theta^{n} \nonumber\\
&=& \frac{1}{n+1} \int_{B(0_{x},r)} \sqrt{-1} \partial s \wedge \overline{\partial} (\tau^{n+1}) \wedge p^{*} \theta^{n} \nonumber\\
&=& - \frac{1}{n+1} \Bigl( \int_{\partial B(0_{x},r)} \tau^{n+1} \sqrt{-1} \partial s \wedge p^{*} \theta^{n} + (2 \pi)^{n+1} c_{1}(L)^{n} \Bigr) \nonumber\\
&=& - C_{1} \int_{\partial B(0_{x},r)} \tau^{n+1} \sqrt{-1} \frac{dw}{w} \wedge p^{*} \theta^{n} - C_{2},
\label{case 1.3}
\end{eqnarray}
where $C_{1}$ and $C_{2}$ are positive  constants depending only on $n$ and $L$.\

In the case $(1)$, previous computations (\ref{tau order1}) and (\ref{case 3}) imply 
\begin{equation}
\tau^{n+1} = O(r^{2(n+1)}).
\end{equation}
For $r> \mbox{diam} (X,\theta)$, the residue theorem holds for each $y \in X$.\
Thus, we have
\begin{equation}
- \int_{\partial B(0_{x},r)} \sqrt{-1} \frac{dw}{w} \wedge p^{*} \theta^{n} = (2 \pi)^{n+1} c_{1}(L)^{n}.
\label{residue formula}
\end{equation}
The formula (\ref{case 1.3}) implies the first statement in Proposition \ref{volume growth}.\

In the case $(2)$, previous computations (\ref{tau order2}) and (\ref{case 3}) imply
\begin{equation}
\tau^{n+1} = O(r^{2}).
\end{equation}
Similarly, (\ref{residue formula}) and the formula (\ref{case 1.3}) imply the second statement in Proposition \ref{volume growth}.\
\sq

\section{
The higher order decay}
\label{sec:3}

In this section, we prove Theorem \ref{scalar curvature decay}.\
Let $(X,L_{X})$ be an $n$-dimensional polarized manifold.\
Let $h_{X}$ be a Hermitian metric on the line bundle $L_{X}$ which defines a {\kah} metric $\theta_{X}$ on $X$.\
Then, the restriction $h_{D}$ of $h_{X}$ to a line bundle $L_{D} := L_{X}|_{D}$ over $D$ defines a {\kah} metric $\theta_{D}$ on $D$.\
Let $\sigma_{D} \in H^{0}(X,L_{X})$ be a defining section of $D$.\
Set $t := \log ||\sigma_{D}||^{-2}$, where $||\sigma_{D}||^{2} = h_{X}(\sigma_{D},\sigma_{D}) $.\
From the construction of the complete {\kah} metrics in Theorem \ref{scalar-flat line} and \cite{BK}, we can define a complete {\kah} metric $\omega_{0}$ on $X \setminus D$ by
\begin{eqnarray*}
\omega_{0}
&:=& \frac{n(n-1)}{\hat{S}_{D}}\dol \exp \left(\frac{\hat{S}_{D}}{n(n-1)}t\right)\\
&=& \exp \left(\frac{\hat{S}_{D}}{n(n-1)}t\right) \left( \theta_{X} + \frac{\hat{S}_{D}}{n(n-1)} \sqrt{-1} \partial t \wedge \overline{\partial} t \right),
\end{eqnarray*}
where $\hat{S}_{D} > 0$ is the average value of the scalar curvature $S(\theta_{D})$:\
$$
\hat{S}_{D} := \frac{\displaystyle\int_{D} S(\theta_{D}) \theta_{D}^{n-1}}{\displaystyle\int_{D} \theta_{D}^{n-1}} = \frac{(n-1) c_{1}(K_{D}^{-1}) \cup c_{1}(L_{D})^{n-2}}{c_{1}(L_{D})^{n-1}}.
$$
By similar ways in Section 2, we have the followings
\begin{lemma}
\label{geo}
Let $r$ be a distance function defined by $\omega_{0}$ from a fixed point $x_{0} \in X \setminus D$.\
Then,
\[
r(x) = O(||\sigma_{D}||^{-\frac{\hat{S}_{D}}{n(n-1)}}(x))
\]
as $x \to D$.\
\end{lemma}
\begin{lemma}
\label{volume growth 2}
The volume growth of $\omega_{0}$ is given by
$$
{\rm Vol}_{\omega_{0}}(B(x_{0},r) ) = O(r^{2n})
$$
as $r \to \infty$.\
\end{lemma}
Thus, Lemma \ref{geo} implies that it is enough to show that
$$
S(\omega_{0}) = O(||\sigma_{D}||^{2+2\hat{S}_{D}/n(n-1)})
$$
as $\sigma_{D} \to 0$.\

To show Theorem \ref{scalar curvature decay}, we have to compute ${\rm Ric}(\theta_{X})$ and ${\rm Ric}(\omega_{0})$.\
Unfortunately, we can't compute the scalar curvature of $\omega_{0}$ in the same way in the proof of Proposition \ref{ricci and scalar}.\
So, we study the determinant of $\theta_{X}$ and the inverse matrix of $\omega_{0}$.\
First, we recall fundamental results in matrix analysis (see \cite{Zhan}).\

\subsection{
Matrix analysis}

To compute Ricci forms of {\kah} metrics $\theta_{X}, \omega_{0}$, we need the following lemma :

\begin{lemma}
\label{determinant}
Consider the following matrix
\begin{eqnarray*}
T = \left[ 
\begin{array}{cc}
A & B \\
C & D \\
\end{array} 
\right],
\end{eqnarray*}
where $A$ is an invertible matrix.\
Then, the determinant of $T$ is given by
$$
\det T = \det A \det (D - C A^{-1} B).
$$
\end{lemma}

The block $D - C A^{-1} B$ is called the Schur complement of the block $D$ of the matrix $T$ (see \cite[p.23]{Zhan}).
For the reader's convenience, we give a proof of this lemma.

\begin{proof}
The result immediately follows from the following formula :
\begin{eqnarray*}
\left[ 
\begin{array}{cc}
A & B \\
C & D \\
\end{array} 
\right]
=
\left[ 
\begin{array}{cc}
I & O \\
C A^{-1} & \dot{I} \\
\end{array} 
\right]
\left[ 
\begin{array}{cc}
A & O \\
O & D - C A^{-1} B \\
\end{array} 
\right]
\left[ 
\begin{array}{cc}
I & A^{-1}B \\
O & \dot{I} \\
\end{array} 
\right]
\end{eqnarray*}
where $I$ and $\dot{I}$ denote suitable identity matrices.
\end{proof}

To take a trace with respect to the {\kah} metric $\omega_{0}$, we need the following inverse matrix formula (see \cite[p.24]{Zhan}) :

\begin{lemma}
\label{inverse formula}
Consider the following matrix
\begin{eqnarray*}
T = \left[ 
\begin{array}{cc}
A & B \\
C & D \\
\end{array} 
\right].
\end{eqnarray*}
Assume that $A$ and $S := D - C A^{-1} B$ are invertible.\
Then, $T$ is invertible and the inverse matrix of $T$ can be written as
\begin{eqnarray*}
T^{-1} = \left[ 
\begin{array}{cc}
A^{-1} + A^{-1} B S^{-1} C A^{-1} & - A^{-1} B S^{-1} \\
- S^{-1} C A^{-1} & S^{-1} \\
\end{array} 
\right].
\end{eqnarray*}
\end{lemma}

Similarly, we give a proof of this lemma for the reader's convenience.

\begin{proof}
From the proof of the previous lemma, we have
\begin{eqnarray*}
\left[ 
\begin{array}{cc}
A & B \\
C & D \\
\end{array} 
\right]^{-1}
&=&
\left[ 
\begin{array}{cc}
I & A^{-1}B \\
O & \dot{I} \\
\end{array} 
\right]^{-1}
\left[ 
\begin{array}{cc}
A & O \\
O & S \\
\end{array} 
\right]^{-1}
\left[ 
\begin{array}{cc}
I & O \\
C A^{-1} & \dot{I} \\
\end{array} 
\right]^{-1}\\
&=&
\left[ 
\begin{array}{cc}
I & - A^{-1}B \\
O & \dot{I} \\
\end{array} 
\right]
\left[ 
\begin{array}{cc}
A^{-1} & O \\
O & S^{-1} \\
\end{array} 
\right]
\left[ 
\begin{array}{cc}
I & O \\
- C A^{-1} & \dot{I} \\
\end{array} 
\right]\\
&=&
\left[ 
\begin{array}{cc}
A^{-1} + A^{-1} B S^{-1} C A^{-1} & - A^{-1} B S^{-1} \\
- S^{-1} C A^{-1} & S^{-1} \\
\end{array} 
\right].
\end{eqnarray*}
\end{proof}

\subsection{
Local trivialization and normal coordinates}

Before studying the scalar curvature $S(\omega_{0})$ near $D$, we choose a local trivialization and normal coordinates around a point of $D$.\

First, fix a point $ p \in D $.\
Since $D$ is the smooth hypersurface of $X$, there exist local holomorphic coordinates $(z^{1},z^{2},...,z^{n-1},w)$ centered at $p$ where $D$ is defined by $ \{ w = 0 \} $ locally and $(z^{1},z^{2},...,z^{n-1})$ are local holomorphic coordinates of $D$.\
Then, there exists a local trivialization of $L_{X}$ such that we can write as $||\sigma_{D}||^{2} = |w|^{2} e^{- \varphi}$ for a smooth function $\varphi$ near $p$ satisfying
\[
d \varphi (0) = 0.
\]
We may assume that if $(z^{1},z^{2},...,z^{n-1},w) = (0,0,...,0,w)$, we have
\begin{equation}
\label{good local coor}
\varphi = O(|w|^{2}).
\end{equation}

Second, we consider the existence of normal coordinates with respect to the {\kah} metric $\theta_{X}$ around $p$ preserving the condition (\ref{good local coor}).\
Since $\theta_{X} = \dol t = \dol \log ||\sigma_{D}||^{-2}$ is the {\kah} metric on $X$, in coordinates above, we can write locally as
\[
\theta_{X} = \sqrt{-1} \left( \sum_{i,j =1}^{n-1} g_{i,\overline{j}} dz^{i} \wedge d\overline{z}^{j} + \sum_{a=1}^{n-1} \left( g_{a,\overline{w}} dz^{a} \wedge d\overline{w} + g_{w,\overline{a}} dw \wedge d\overline{z}^{a} \right) + g_{w,\overline{w}} dw \wedge d\overline{w} \right).
\]
For simplicity, write $(z^{1},...,z^{n-1},w) = (z;w)$.\
Consider another holomorphic coordinate chart $(\hat{z}^{1},...,\hat{z}^{n-1},w) = (\hat{z};w)$ around $p \in D$.\
Directly, we have
\begin{eqnarray*}
\frac{\partial}{\partial w} \left( g_{\hat{i}, \overline{\hat{j}}} \right)
&=&
\frac{\partial}{\partial w} \left( \frac{\partial z^{k}}{\partial \hat{z}^{i}} \frac{\partial \overline{z}^{l}}{\partial \overline{\hat{z}}^{j}}  g_{k, \overline{l}} \right)\\
&=&
\frac{\partial}{\partial w} \frac{\partial z^{k}}{\partial \hat{z}^{i}}  \left( \frac{\partial \overline{z}^{l}}{\partial \overline{\hat{z}}^{j}}  g_{k, \overline{l}} \right) + \frac{\partial z^{k}}{\partial \hat{z}^{i}} \frac{\partial \overline{z}^{l}}{\partial \overline{\hat{z}}^{j}} \frac{\partial g_{k, \overline{l}}}{\partial w}.
\end{eqnarray*}
Set the condition
\[
\frac{\partial z^{k}}{\partial \hat{z}^{i}}(0;0) = \delta_{k,i}.
\]
So, we have
\[
\frac{\partial}{\partial w} \left( g_{\hat{i}, \overline{\hat{j}}} \right) (0;0) = \frac{\partial}{\partial w} \frac{\partial z^{k}}{\partial \hat{z}^{i}} g_{k, \overline{j}}(0;0) + \frac{\partial g_{i, \overline{j}}}{\partial w}(0;0).
\]
Considering the equation $\partial g_{\hat{i}, \overline{\hat{j}}} / \partial w (0;0) = 0$, we have
\[
\frac{\partial}{\partial w} \frac{\partial z^{k}}{\partial \hat{z}^{i}}(0;0) = - \sum_{j} g^{k, \overline{j}}(0;0) \frac{\partial g_{i, \overline{j}}}{\partial w}(0;0).
\]
Thus, we have
\begin{lemma}
By the change of holomorphic coordinates $(\hat{z};w)$ around $p \in D$ defined by
\[
z^{\alpha} = \sum_{i=1}^{n-1} \hat{z}^{i} \left(\delta_{i,\alpha} - w \sum_{j=1}^{n-1} g^{\alpha,\overline{j}}(0;0) \frac{\partial g_{i,\overline{j}}}{\partial w}(0;0) \right)\hspace{7pt} (\alpha = 1,2,...,n-1),
\]
we have
\begin{equation}
\label{good local coor1}
\frac{\partial g_{\hat{i},\overline{\hat{j}}}}{\partial w}(0;0) = 0.
\end{equation}
In particular, at $(\hat{z};w) = (0;w)$, we have
\begin{equation}
\label{good local coor2}
g_{\hat{i},\overline{\hat{j}}}(0;w) = g_{\hat{i},\overline{\hat{j}}}(0;0) + O(|w|^{2}).
\end{equation}
\end{lemma}
Consequently, we obtain
\begin{prop}
We can find a local trivialization of $L_{X}$ and local holomorphic coordinates so that
\[
\varphi = O(|w|^{2}) , \hspace{8pt} g_{\hat{i},\overline{\hat{j}}}(0;w) = g_{\hat{i},\overline{\hat{j}}}(0;0) + O(|w|^{2}).
\]
at $(\hat{z}^{1}, ... ,\hat{z}^{n-1}, w) = (0,...,0,w)$.\
\end{prop}

\begin{proof}
In new local coordinates above, we have
\[
\frac{\partial \varphi}{\partial \hat{z}^{i}} = \frac{\partial \varphi}{\partial z^{j}} \frac{\partial z^{j}}{\partial \hat{z}^{i}} + \frac{\partial \varphi}{\partial w} \frac{\partial w}{\partial \hat{z}^{i}},
\]
and
\[
\frac{\partial z^{j}}{\partial \hat{z}^{i}}(0;0) = \delta_{i,j} , \hspace{7pt} \frac{\partial \varphi}{\partial z^{i}}(0;0) = \frac{\partial \varphi}{\partial w}(0;0) = 0.
\]
Thus, the proposition follows.
\end{proof}
For simplicity, we write new local coordinates $(\hat{z}^{1}, ... ,\hat{z}^{n-1}, w)$ by the same symbol $(z^{1}, ... ,z^{n-1}, w)$.\

\subsection{
Proof of Theorem \ref{scalar curvature decay}}

Recall that $\omega_{0}$ is written as
\[
\omega_{0} = \exp \left(\frac{\hat{S}_{D}}{n(n-1)}t\right) \left( \theta_{X} + \frac{\hat{S}_{D}}{n(n-1)} \sqrt{-1} \partial t \wedge \overline{\partial}t \right)
\]
and it is enough to show that
$$
S(\omega_{0}) = O(||\sigma_{D}||^{2+2\hat{S}_{D}/n(n-1)})
$$
as $\sigma_{D} \to 0$.\
First, we show
\begin{lemma}
\label{background ricci}
The Ricci form of $\omega_{0}$ is given by
$$
{\rm Ric }(\omega_{0}) = {\rm Ric }(\theta_{X}) - \frac{\hat{S}_{D}}{n-1} \theta_{X} - \dol \log \left( 1 + \frac{\hat{S}_{D}}{n(n-1)} || \partial t ||^{2}_{\theta_{X}} \right).
$$
\end{lemma}

\begin{proof}
To prove this lemma, it is enough to see the volume form of $\omega_{0}$.\
From the definition of $\omega_{0}$, we have
$$
\omega_{0} = \exp \left(\frac{\hat{S}_{D}}{n(n-1)}t\right) \left( \theta_{X} + \frac{\hat{S}_{D}}{n(n-1)} \sqrt{-1} \partial t \wedge \overline{\partial} t \right).
$$
So, the following identity
$$
\sqrt{-1} \partial t \wedge \overline{\partial} t \wedge \theta_{X}^{n-1} = \frac{1}{n}  || \partial t ||^{2}_{\theta_{X}} \theta_{X}^{n}
$$
implies that the volume form of $\omega_{0}$ is given by
$$
\omega_{0}^{n} = \exp \left(\frac{\hat{S}_{D}}{n-1}t\right) \left( 1 + \frac{\hat{S}_{D}}{n(n-1)} || \partial t ||^{2}_{\theta_{X}} \right) \theta_{X}^{n}.
$$
Recall that the Ricci form is given by ${\rm Ric}(\omega_{0}) = - \dol \log \omega_{0}^{n}$.\
Thus, the lemma follows.
\end{proof}

Thus, we easily have $S(\omega_{0}) = O(||\sigma_{D}||^{2\hat{S}_{D}/n(n-1)})$ as $\sigma_{D} \to 0$.\

Firstly, we show the following proposition to prove Theorem \ref{scalar curvature decay}.

\begin{prop}
\label{1 decay}
If $\theta_D$ is a cscK metric, we have
$$
S(\omega_0) = O(|| \sigma_D ||^{1 + 2\hat{S}_{D}/n(n-1)})
$$
as $\sigma_D \to 0$.
\end{prop}
\begin{proof}
To prove this, we compute the Ricci form of $\theta_{X}$.\
Write
\begin{eqnarray*}
\theta_{X} = \left[ 
\begin{array}{cccc}
g_{1,\overline{1}} & \cdots & g_{1,\overline{n-1}} & g_{1,\overline{w}} \\
\vdots & \ddots & \vdots & \vdots \\
g_{n-1,\overline{1}} & \cdots & g_{n-1,\overline{n-1}} & g_{n-1,\overline{w}}\\
g_{w,\overline{1}} & \cdots & g_{w,\overline{n-1}} & g_{w,\overline{w}}
\end{array} 
\right]
=
\left[ 
\begin{array}{cc}
B & R \\
\overline{R}^{t} & W \\
\end{array} 
\right]
\end{eqnarray*}
in the previous local holomorphic coordinates.\
Since Lemma \ref{determinant} implies that $\det \theta_{X} = \det B \det ( W - \overline{R}^{t} B^{-1} R )$, we have
\begin{eqnarray*}
{\rm Ric}(\theta_{X})
&=& -\dol \log \det B -\dol \log ( W - \overline{R}^{t} B^{-1} R ).
\end{eqnarray*}
Recall the notation $(z^{1},...,z^{n-1},w) = (z;w)$.\
Consider the expansion at $w=0$ ;
\begin{equation}
\label{det order}
\det B (z;w) = \det B(z;0) + w \frac{\partial \det B}{\partial w} + \overline{w} \frac{\partial \det B}{\partial \overline{w}} + O(|w|^{2}).
\end{equation}

Recall that
\[
{\rm Ric} (\theta_{D}) = - \sqrt{-1} \sum_{i,j = 1}^{n-1} \frac{\partial^{2}\log \det B(z;0) }{\partial z^{i} \partial \overline{z}^{j} } dz^{i} \wedge d\overline{z}^{j},
\]
and $S(\theta_{D}) = {\rm tr}_{\theta_{D}} {\rm Ric}(\theta_{D} ) = \hat{S}_{D}$.\
By (\ref{det order}),
\begin{eqnarray*}
{\rm Ric}(\theta_{X})
&=&{\rm Ric} (\theta_{D}) + O(|w|) dz \wedge d\overline{z} -\dol \log ( W - \overline{R}^{t} B^{-1} R )\\
&+&O(1) dw \wedge d\overline{z} + O(1) dz \wedge d\overline{w} + O(1) dw \wedge d\overline{w}
\end{eqnarray*}
at $ (0;w)$.\
Here $dz$ denote differential 1-forms in directions of $D$.\
To prove Proposition \ref{1 decay}, it is clearly enough to take the trace with respect to the metric
\[
\theta_{X} + \frac{\hat{S}_{D}}{n(n-1)} \sqrt{-1} \partial t \wedge \overline{\partial}t.
\]
For simplicity, set $a := \hat{S}_{D}/n(n-1) > 0$.\
Since $\partial t = \partial \varphi - dw/w$, the metric $\theta_{X} + a \sqrt{-1} \partial t \wedge \overline{\partial}t$ can be written as
\begin{eqnarray*}
\left[ 
\begin{array}{cccc}
g_{1,\overline{1}} + a \varphi_{1} \varphi_{\overline{1}}  & \cdots & g_{1,\overline{n-1}} + a \varphi_{1} \varphi_{\overline{n-1}}& g_{1,\overline{w}} + a \varphi_{1} (\varphi_{\overline{w}} - 1/\overline{w} ) \\
\vdots & \ddots & \vdots & \vdots \\
g_{n-1,\overline{1}} + a \varphi_{n-1} \varphi_{\overline{1}} & \cdots & g_{n-1,\overline{n-1}} + a \varphi_{n-1} \varphi_{\overline{n-1}}& g_{n-1,\overline{w}} + a \varphi_{n-1} (\varphi_{\overline{w}} - 1/\overline{w} )\\
g_{w,\overline{1}} + a (\varphi_{w} - 1/w) \varphi_{\overline{1}} & \cdots & g_{w,\overline{n-1}} + a (\varphi_{w} - 1/w) \varphi_{\overline{n-1}} & g_{w,\overline{w}} + a(\varphi_{w} - 1/w )(\varphi_{\overline{w}} - 1/\overline{w} )
\end{array} 
\right],
\end{eqnarray*}
where $\varphi_{i}$ denotes $\partial \varphi / \partial z^{i}$.\
For simplicity, write the matrix above as
\[
\theta_{X} + a \sqrt{-1} \partial t \wedge \overline{\partial}t =
\left[
\begin{array}{cc}
E & F\\
G & H
\end{array}
\right].
\]
In order to take the trace of ${\rm Ric}(\omega_{0})$ with respect to the metric $\theta_{X} + a \sqrt{-1} \partial t \wedge \overline{\partial}t$, we compute the inverse matrix of this.\
Since we only consider $S(\omega_{0})$ near $D$, $H = O(|w|^{-2})$ as $w \to 0$.\ 
By Lemma \ref{inverse formula}, we have
\begin{eqnarray*}
\left[ 
\begin{array}{cc}
E^{-1} + E^{-1} F S^{-1} G E^{-1} & - E^{-1} F S^{-1} \\
- S^{-1} G E^{-1} & S^{-1} \\
\end{array} 
\right],
\end{eqnarray*}
where $S := H - G E^{-1} F$.\
Since $S = O(|w|^{-2})$ as $w \to 0$, we get
$$
E^{-1} F S^{-1} G E^{-1},E^{-1} F S^{-1}, S^{-1} G E^{-1},S^{-1} = O(|w|^{2}).
$$
Thus, to compute the scalar curvature $S(\omega_{0})$, it is enough to study the block $E^{-1} + E^{-1} F S^{-1} G E^{-1}$.\
In this case, by considering the expansion at $w=0$, we can write
\[
E = B(0;0) + J,
\]
where $J=O(|w|^{2})$.\
So we have
\begin{eqnarray*}
E^{-1}
&=& (B(0;0)+J)^{-1}\\
&=& B(0;0)^{-1} (I +J B(0;0)^{-1})^{-1}\\
&=& B(0;0)^{-1} (I + \sum_{i>0} (- J B(0;0)^{-1})^{i})\\
&=& B(0;0)^{-1} + O(|w|^{2}).
\end{eqnarray*}
Consider the term
\[
- \dol \log \left( 1 + \frac{\hat{S}_{D}}{n(n-1)} || \partial t ||_{\theta_{X}}^{2} \right),
\]
where
\begin{eqnarray*}
|| \partial t ||_{\theta_{X}}^{2}
&=& \sum_{i,j}^{n-1} g^{i,\overline{j}} \varphi_{i} \varphi_{\overline{j}} + \sum_{a=1}^{n-1}  \left( g^{a,\overline{w}} \varphi_{a} ( \varphi_{\overline{w}} - 1/\overline{w} ) + g^{w,\overline{a}} ( \varphi_{w} - 1/w ) \varphi_{\overline{a}} \right)\\
&\hspace{30pt} +& g^{w,\overline{w}} ( \varphi_{w} - 1/w ) ( \varphi_{\overline{w}} - 1/\overline{w} ).
\end{eqnarray*}
Note that $ g^{w,\overline{w}} = ( W - \overline{R}^{t} B^{-1} R )^{-1} $.\
Thus, we have
\begin{eqnarray*}
&&-\dol \log ( W - \overline{R}^{t} B^{-1} R ) - \dol \log \left( 1 + \frac{\hat{S}_{D}}{n(n-1)} || \partial t ||_{\theta_{X}}^{2} \right)\\
&=& -\dol \log ( 1 + O(|w|^{2}) ).
\end{eqnarray*}
Thus,
\begin{eqnarray*}
|| \sigma_{D} ||^{-2a} S(\omega_{0})
&=& {\rm tr}_{\theta_{X} + a \sqrt{-1} \partial t \wedge \overline{\partial}t } {\rm Ric}(\omega_{0} )\\
&=& {\rm tr}_{\theta_{X} + a \sqrt{-1} \partial t \wedge \overline{\partial}t } \left( {\rm Ric} (\theta_{D}) - \frac{\hat{S}_{D}}{(n-1)} \theta_{D} \right) + O(|w|)
\end{eqnarray*}
as $w \to 0$.\
Therefore, Proposition \ref{1 decay} is proved.
\ 
\end{proof}

\begin{rem}
Roughly, we have proved that
\[
S(\omega_{0}) = C ||\sigma_{D}||^{\frac{2 \hat{S}_{D}}{n(n-1)}} (S(\theta_{D}) - \hat{S}_{D} + O(||\sigma_{D}||))
\]
near $D$.\
Thus, in fact, $\theta_{D}$ is cscK if and only if $S(\omega_{0})$ has a zero along $D$ of order $1 + 2\hat{S}_{D}/n(n-1)$ in our construction.\ 
\end{rem}

Secondly, we prove Theorem \ref{scalar curvature decay} by altering the Hermitian metric $h_X$.
By following Bando-Kobayashi \cite{BK}, we take a smooth function $a \in C^{\infty} (X, \mathbb{R})$ such that $a |_D \equiv 0$.
Define a Hermitian metric on $L_X$ by
$$
h_{X,a} := e^{-a} h_X.
$$
Note that this modification does not change the Hermitian metric $h_D$ on $L_D$.
For this Hermitian metric $h_{X,a}$, we write
$$
|| \sigma_D ||^{2}_a := || \sigma_D ||^{2}_{h_{X,a}} = e^{-a} || \sigma_D ||^{2}_{h_{X}}, \hspace{7pt} t_a = \log || \sigma_D ||^{-2}_a.
$$
In addition, we define the {\kah} metrics by
$$
\theta_{X,a} := \dol t_a, \hspace{7pt} \omega_a := \frac{n(n-1)}{\hat{S}_{D}} \dol \exp \left( \frac{\hat{S}_{D}}{n(n-1)} t_a \right).
$$
We consider the following function :
$$
|| \sigma_D ||^{-2 \hat{S}_{D}/n(n-1)}_a S(\omega_a).
$$

Take a point $p \in D$ and a local holomorphic coordinate chart centered at $p$ such that $D = \{ w = 0 \}$.
In order to prove Theorem \ref{scalar curvature decay}, it is enough to show the following proposition :
\begin{prop}
We can find a smooth function $a$ on $X$ such that $a|_D \equiv 0$ and
$$
\frac{\partial}{\partial w } \left( || \sigma_D ||^{-2 \hat{S}_{D}/n(n-1)}_a S(\omega_a) \right) = 0
$$
at any point $p$.
\end{prop}
\begin{proof}
From Proposition \ref{1 decay}, we can find functions $F_0$ and $F_a$ such that
$$
\Delta_{\omega_0} F_0 = S(\omega_0), \hspace{7pt} \Delta_{\omega_a} F_a = S(\omega_a)
$$
and $F_0$ and $F_a$ decay near $D$ by following \cite{BK} (see also Lemma \ref{isomorphic Laplas} in this paper).
By using these functions $F_0 , F_a$, we have
\begin{eqnarray*}
\Delta_{\omega_a} F_a
&=&S(\omega_a) \\
&=& {\rm tr}_{\omega_a} {\rm Ric }(\omega_a)\\
&=& {\rm tr}_{\omega_a} ( {\rm Ric }(\omega_a) -  {\rm Ric }(\omega_0) ) + ( {\rm tr}_{\omega_a} -  {\rm tr}_{\omega_0})  {\rm Ric }(\omega_0) + \Delta_{\omega_0} F_0\\
&=& - {\rm tr}_{\omega_a} \dol \log \left( \frac{ \omega_a^n }{\omega_0^n} \right) + ( {\rm tr}_{\omega_a} -  {\rm tr}_{\omega_0})  {\rm Ric }(\omega_0) + \Delta_{\omega_0} F_0.
\end{eqnarray*}
At $p \in D$, we have the following equation by the direct computation :
\begin{eqnarray*}
\Delta_{\theta_D} \frac{\partial F_a}{\partial w}
&=&\frac{\partial}{\partial w } \left( || \sigma_D ||^{-2 \hat{S}_{D}/n(n-1)}_a S(\omega_a) \right) \\
&=& - \Delta^2_{\theta_D} \frac{\partial a}{\partial w} - \left( \frac{\hat{S}_{D}}{n-1} -1 \right) \Delta_{\theta_D} \frac{\partial a}{\partial w}\\
&& \hspace{20pt}  - \left(\dol \frac{\partial a}{\partial w}, \hspace{3pt} {\rm Ric} \theta_D - \frac{\hat{S}_{D}}{n-1} \theta_D \right)_{\theta_D} + \Delta_{\theta_D} \frac{\partial F_0}{\partial w}.
\end{eqnarray*}
Recall that the linearization of the scalar curvature operator satisfies
$$
L_{\theta_D} \varphi = - \Delta^2_{\theta_D} \varphi - \left(\dol \varphi, \hspace{3pt} {\rm Ric} \theta_D \right)_{\theta_D} = - \mathcal{D}^{*}_{\theta_D}\mathcal{D}_{\theta_D} \varphi + (\nabla^{1,0} \varphi, \nabla^{0,1} S(\theta_D))_{\theta_D}
$$
for $\varphi \in C^\infty (D)$ (see \cite{Sz1}).
So, we have
\begin{eqnarray*}
\frac{\partial}{\partial w } \left( || \sigma_D ||^{-2 \hat{S}_{D}/n(n-1)}_a S(\omega_a) \right)
&=& L_{\theta_D} \frac{\partial a}{\partial w} + \Delta_{\theta_D} \frac{\partial a}{\partial w} + \Delta_{\theta_D} \frac{\partial F_0}{\partial w}\\
&=& - \mathcal{D}^{*}_{\theta_D}\mathcal{D}_{\theta_D} \frac{\partial a}{\partial w}  + \Delta_{\theta_D} \frac{\partial a}{\partial w} + \Delta_{\theta_D} \frac{\partial F_0}{\partial w}.
\end{eqnarray*}
Here, we have used the fact that the {\kah} metric $\theta_D$ is a cscK metric on $D$.
Recall that the Laplacian with respect to $\theta_D$ is given by $\Delta_{\theta_D} = g^{i,\overline{j}} \partial_i \partial_{\overline{j}}$.
Thus, by the Hodge theory, we can solve the following differential equation :
\begin{equation}
\label{our modification}
- \mathcal{D}^{*}_{\theta_D}\mathcal{D}_{\theta_D} \frac{\partial a}{\partial w}  + \Delta_{\theta_D} \frac{\partial a}{\partial w} + \Delta_{\theta_D} \frac{\partial F_0}{\partial w} = \Delta_{\theta_D} \frac{\partial F_a}{\partial w} = 0.
\end{equation}
Note that the operator $- \mathcal{D}^{*}_{\theta_D}\mathcal{D}_{\theta_D} + \Delta_{\theta_D} $ is negative and self-adjoint.
By the same way in \cite[p,176]{BK}, we can show the existence of a smooth function $a \in C^\infty (X, \mathbb{R})$ such that $a|_D \equiv 0$ and $\theta_{X,a} = \theta_X + \dol a > 0$ on $X$.
\ 
\end{proof}

\begin{rem}
In \cite[p,176]{BK}, if $\theta_{D}$ is a Ricci-positive {\kah}-Einstein metric, the background {\kah} metric $\omega_{0}$ can be chosen so that the Ricci potential of $\omega_{0}$ decays at a higher order by altering the Hermitian metric $h_{X}$ on $K_{X}^{-1/\alpha}$.\
In order to find the Hermitian metric above, they solved the following differential equation :
\begin{equation}
\label{BK modification}
\Delta_{\theta_D} \frac{\partial a}{\partial w} + (\alpha - 2) \frac{\partial a}{\partial w} + \frac{\partial F_0}{\partial w} = 0.
\end{equation}
Here, we have used the notations in this article.
In the case of Bando-Kobayashi \cite{BK}, the {\kah} metric $\theta_D$ is a {\kah} Einstein metric, i.e., ${\rm Ric } (\theta_D) = (\alpha - 1) \theta_D$, so this equation (\ref{BK modification}) is equivalent to the equation (\ref{our modification}) by considering the image of the operator $- \Delta_{\theta_D}$ of (\ref{BK modification}).
Therefore, our modification of the Hermitian metric $h_X$ can be considered as a generalization of the modification in Bando-Kobayashi \cite{BK}.
\end{rem}

From now on, let us write the modified Hermitian metric $h_{X,a}$ as $h_X$ for simplicity.
So, we use the simple symbols $t, \theta_X, \omega_0$ from now on.
Thus, we have $S(\omega_0) = O(|| \sigma_D ||^{2 + 2\hat{S}_{D}/n(n-1)})$ on the assumption that $\theta_D$ is cscK.

\section{
Asymptotically conical geometry}
\label{sec:4}

Recall that the {\kah} metric defined by
\begin{eqnarray*}
\omega_{0}
&=& \frac{n(n-1)}{\hat{S}_{D}}\dol \exp \left(\frac{\hat{S}_{D}}{n(n-1)}t\right)\\
&=& \exp \left(\frac{\hat{S}_{D}}{n(n-1)}t\right) \left( \theta_{X} + \frac{\hat{S}_{D}}{n(n-1)} \sqrt{-1} \partial t \wedge \overline{\partial} t \right)
\end{eqnarray*}
is complete on $X \setminus D$.\
Set $r(x) := d(x,x_{0})$, where $d$ is the distance function from some fixed point $x_{0} \in X\setminus D$ defined by $\omega_{0}$.\
Following \cite{BK}, the Riemannian manifold $(X\setminus D , \omega_{0})$ is of asymptotically conical geometry which is the analytic framework in this paper.\

\begin{definition}
A complete Riemannian metric $g$ on an open manifold $M$ of dimension $m$ is said to be of {\it $C^{k,\alpha}$-asymptotically conical geometry} if for each point $p \in M$ with distance $r$ from a fixed point $o \in M$, there exists a harmonic coordinate system $x = (x^{1},x^{2},\cdot\cdot\cdot,x^{m})$ centered at $p$ which satisfies the following conditions:
\begin{itemize}
\item The coordinate $x$ runs over a unit ball $B_{p}^{m} \subset \mathbb{R}^{m}$.
\item If we write $g = \sum g_{i,j}(x) dx^{i} dx^{j} $, then the matrix $(r^{2} + 1)^{-1} g_{i,j}(x)$ is bounded from below by a constant positive matrix independent of $p$.
\item The $C^{k,\alpha}$-norms of $(r^{2} + 1)^{-1} g_{i,j}(x)$ are uniformly bounded.
\end{itemize}
In particular, we simply say that $(M,g)$ is of {\it asymptotically conical geometry} if $(M,g)$ is of $C^{k,\alpha}$-asymptotically conical geometry for any $k \in \mathbb{Z}_{\geq 0}$ and $\alpha \in (0,1)$.\
\end{definition}

\begin{definition}
\label{weight norm}
Assume that a Riemannian manifold $(M,g)$ is of asymptotically conical geometry.\
The $C^{k,\alpha}$-norm of a function $u$ of weight $\delta \in \mathbb{R}$ is defined by $$||u||_{C^{k,\alpha}_{\delta}} := \sup_{p \in M} (r(p)^{2} + 1)^{\delta/2} || u ||_{C^{k,\alpha}(B_{p})}.$$
The Banach space $C^{k,\alpha}_{\delta}$ is defined by the set of functions $u$ such that $||u||_{C^{k,\alpha}_{\delta}} < \infty$.\
In the above definition, we use the coordinates $x \in B_{p}^{m}$ centered at $p$ with $d(o,p) = r$ in the definition of the asymptotically conicalness.\
\end{definition}

\section{
Forth order elliptic linear operators}
\label{sec:5}

To prove Theorem \ref{complete scalar-flat}, we study the linearization of the scalar curvature operator.\
For a smooth function $\varphi$ on $X \setminus D$, set $\omega_{t} := \omega_{0} + t \dol \varphi$.\
Recall that $S(\omega_{t}) = g^{i,\overline{j}}_{t} R_{t,i,\overline{j}}$.\
Thus, the linearization of the scalar curvature operator is defined by 
\begin{eqnarray*}
L_{\omega_{0}} (\phi)
&:=& \left. \frac{d}{dt} \right|_{t=0} S(\omega_{t})\\
&=& - \Delta_{\omega_{0}}^{2} \varphi - g^{i,\overline{q}} \varphi_{p,\overline{q}} g^{p,\overline{j}} R_{j,\overline{i}} \\
&=& - \Delta_{\omega_{0}}^{2} \varphi - R^{i,\overline{j}} \varphi_{i,\overline{j}}.
\end{eqnarray*}

Set $M := X \setminus D$.\
The following operator plays an important role in this article.\
\begin{definition}
The operator $\mathcal{D}_{\omega_{0}}$ is defined by
\begin{eqnarray*}
\mathcal{D}_{\omega_{0}} : C_{\delta}^{k,\alpha}(M,\mathbb{C}) &\to& C_{\delta + 2}^{k-2,\alpha}(M,\Omega^{0,1}M \otimes T^{1,0}M)\\
\varphi &\mapsto& \overline{\partial} (\nabla^{1,0} \varphi)
\end{eqnarray*}
\end{definition}
Here $\overline{\partial}$ is the (0,1)-part of the Levi-Civita connection and $\nabla^{1,0}$ is the (1,0)-gradient with respect to $\omega_{0}$.\
We call $\mathcal{D}^{*}_{\omega_{0}}\mathcal{D}_{\omega_{0}}$ the Lichnerowicz operator and we have

\begin{lemma}
The Lichnerowicz operator $\mathcal{D}^{*}_{\omega_{0}}\mathcal{D}_{\omega_{0}}$ satisfies
\begin{eqnarray}
\mathcal{D}^{*}_{\omega_{0}}\mathcal{D}_{\omega_{0}} \varphi = \Delta_{\omega_{0}}^{2} \varphi + R^{i,\overline{j}} \varphi_{i,\overline{j}} + (\nabla^{1,0} \varphi, \nabla^{0,1} S(\omega_{0}))_{\omega_{0}}.
\end{eqnarray}
\end{lemma}

Thus, we have
$$
L_{\omega_{0}} = - \mathcal{D}^{*}_{\omega_{0}}\mathcal{D}_{\omega_{0}} + (\nabla^{1,0} \ast, \nabla^{0,1} S(\omega_{0}))_{\omega_{0}}.
$$

The idea of proving Theorem \ref{complete scalar-flat} follows from Arezzo-Pacard \cite{AP1} and \cite{AP2} (see also \cite{Sz1}).\
Consider the following expansion :
$$
S(\omega_{0} + \dol \phi) = S(\omega_{0}) + L_{\omega_{0}}(\phi) + Q_{\omega_{0}}(\phi).
$$
To solve the following equation ;
\[
S(\omega_{0} + \dol \phi) = 0,
\]
we will find a following fixed point :
\[
\phi = - L_{\omega_{0}}^{-1}(S(\omega_{0}) + Q_{\omega_{0}}(\phi)).
\]
When we prove Theorem \ref{complete scalar-flat}, we assume that $L_{\omega_{0}}$ is invertible.\
Therefore we need to prove that the operator
\begin{equation}
\label{operator N}
\mathcal{N}(\phi) := - L_{\omega_{0}}^{-1}(S(\omega_{0}) + Q_{\omega_{0}}(\phi))
\end{equation}
is a contraction on some Banach space.\

In particular, we mainly use the weighted Banach spaces $C^{4,\alpha}_{\delta - 4}(X\setminus D)$ and $C^{0,\alpha}_{\delta}(X\setminus D)$.\
From the definition of the weighted Banach space and local formulae of these operators, we easily have
\begin{lemma}
Following three operators
\[
L_{\omega_{0}} ,\hspace{4pt} \mathcal{D}_{\omega_{0}}^{*}\mathcal{D}_{\omega_{0}} ,\hspace{4pt} \Delta_{\omega_{0}}^{2} :C^{4,\alpha}_{\delta - 4}(X \setminus D) \to C^{0,\alpha}_{\delta}(X \setminus D)
\]
are bounded.\
\end{lemma}

First, we study the square of the Laplacian operator $\Delta_{\omega_{0}}^{2}$.\
Define a barrier function $\rho$ on $X \setminus D$ by
\[
\rho := \exp \left(\frac{\hat{S}_{D}}{2n(n-1)}t\right) = ||\sigma_{D}||^{-\hat{S}_{D}/n(n-1)}.
\]
Note that for $\delta > 0$, $\rho$ satisfies
\begin{eqnarray*}
\Delta_{\omega_{0}} \rho^{-\delta} 
&=& {\rm tr}_{\omega_{0}} \dol \exp \left(\frac{ - \delta \hat{S}_{D}}{2n(n-1)}t\right)\\
&=& \frac{ - \delta \hat{S}_{D}}{2n(n-1)} {\rm tr}_{\omega_{0}} \left( \exp \left(\frac{ - \delta \hat{S}_{D}}{2n(n-1)}t\right) \left( \dol t + \frac{ - \delta \hat{S}_{D}}{2n(n-1)} \sqrt{-1} \partial t \wedge \overline{\partial} t \right) \right)\\
&=& \frac{ - \delta \hat{S}_{D}}{2n(n-1)} \rho^{-\delta - 2} {\rm tr}_{\omega_{0}} \left( \exp \left(\frac{\hat{S}_{D}}{n(n-1)}t\right) \left( \dol t + \frac{ - \delta \hat{S}_{D}}{2n(n-1)} \sqrt{-1} \partial t \wedge \overline{\partial} t \right) \right)\\
&=& \frac{ - \delta \hat{S}_{D}}{2n(n-1)} \rho^{-\delta - 2} {\rm tr}_{\omega_{0}} \left(\omega_{0} + \frac{ - (\delta + 2 )\hat{S}_{D}}{2n(n-1)} \exp \left(\frac{\hat{S}_{D}}{n(n-1)}t\right) \sqrt{-1} \partial t \wedge \overline{\partial} t \right)\\
&=& \frac{ - \delta \hat{S}_{D}}{2n(n-1)} \rho^{-\delta - 2} \left( n + {\rm tr}_{\omega_{0}} \left( \frac{ - (\delta + 2 )\hat{S}_{D}}{2n(n-1)} \exp \left(\frac{\hat{S}_{D}}{n(n-1)}t\right) \sqrt{-1} \partial t \wedge \overline{\partial} t \right) \right)\\
&\leq& \frac{ - \delta \hat{S}_{D}}{2n(n-1)} \rho^{-\delta - 2} \left( n - \frac{\delta + 2}{2} \right)\\
\end{eqnarray*}
Here we have used the following inequality :
\[
\omega_{0} \geq \frac{\hat{S}_{D}}{n(n-1)} \exp \left(\frac{\hat{S}_{D}}{n(n-1)}t\right) \sqrt{-1} \partial t \wedge \overline{\partial} t.
\]
From Lemma \ref{volume growth 2}, we have known that the volume growth of $\omega_{0}$ is given by
\[
{\rm Vol }_{\omega_{0}} (B(x_{0},r)) = O(r^{2n}).
\]
In addition, we have known that $|| {\rm Ric}(\omega_{0}) ||_{\omega_{0}} = O(r^{-2})$ as $r \to \infty$.\
From \cite[Theorem 1.2]{He}, we have
\begin{lemma}
Set $\gamma := n/(n-1)$.\
Then the following Sobolev inequality holds, i.e., there exists a constant $C>0$ such that
\[
\left( \int_{X \setminus D} |v|^{2\gamma} \omega_{0}^{n} \right)^{1/\gamma} \leq C \int_{X \setminus D} |\partial v|^{2} \omega_{0}^{n}
\]
for any compactly supported smooth function $v$ on $X \setminus D$.\
\end{lemma}
Then, we can apply the Moser's iteration to obtain the $C^{0}_{\delta}$-estimate.\
Following \cite[p,178]{BK}, we have
\begin{lemma}
\label{isomorphic Laplas}
If $2 < \delta < 2n$, the Laplacian $\Delta_{\omega_{0}}:C^{k,\alpha}_{\delta - 2}(X\setminus D) \to C^{k-2,\alpha}_{\delta}(X\setminus D)$ is isomorphic.\
\end{lemma}

Recall that the standard theorem on Banach spaces (see \cite[p.77]{Yosida}).\

\begin{thm}
\label{range theorem}
Let $\mathcal{X}$ and $\mathcal{Y}$ be Banach spaces.\
Assume that $L : \mathcal{X} \to \mathcal{Y}$ is a bounded and isomorphic linear operator.\
Then, the inverse $L^{-1}$ is also bounded.\
\end{thm}

Thus, the inverse of the Laplacian $\Delta_{\omega_{0}}^{-1}$ is bounded.\
In addition, recall the definition of Fredholm operators (see \cite[Chapter 1, \S 1.4]{Gil}):
\begin{definition}
We say that a bounded linear operator $L : \mathcal{X} \to \mathcal{Y}$ between Banach spaces $\mathcal{X}$ and $\mathcal{Y}$ is a {\it Fredholm operator} if the ${\rm dim (Ker}L)$ and ${\rm dim (Coker } L)$ are finite and ${\rm Im}L$ is a closed linear subspace of $\mathcal{Y}$.\
For such an operator $L$, we define an {\it index} of $L$ by
$$
{\rm ind} (L) := {\rm dim (Ker}L) - {\rm dim (Coker } L).
$$
\end{definition}

Thus, immediately we obtain
\begin{lemma}
If $2 < \delta < 2n $, the Laplacian $\Delta_{\omega_{0}}:C^{k,\alpha}_{\delta - 2}(X\setminus D) \to C^{k-2,\alpha}_{\delta}(X\setminus D)$ is a Fredholm operator whose index ${\rm ind}(\Delta_{\omega_{0}})$ is zero.\
Moreover, there exists a bounded inverse $\Delta_{\omega_{0}}^{-1}$ which is also a Fredholm operator whose index is zero.\
\end{lemma}

Next, we study the operator $\mathcal{D}_{\omega_{0}}^{*}\mathcal{D}_{\omega_{0}}$.\

\begin{lemma}
\label{injective}
Assume that $\delta > 4$ and there is no nonzero holomorphic vector field on $X$ which vanishes on $D$.\
Then, the operator
$$
\mathcal{D}_{\omega_{0}}^{*}\mathcal{D}_{\omega_{0}}:C^{4,\alpha}_{\delta - 4}(X\setminus D) \to C^{0,\alpha}_{\delta}(X\setminus D)
$$
is injective.\
\end{lemma}

\begin{proof}
Assume that $\phi \in C^{4,\alpha}_{\delta - 4}(X\setminus D)$ satisfies $\mathcal{D}_{\omega_{0}}^{*}\mathcal{D}_{\omega_{0}}\phi = 0$.\
Integrating by parts, we have
\[
0 = \int_{X \setminus D} \phi \mathcal{D}_{\omega_{0}}^{*}\mathcal{D}_{\omega_{0}}\phi \omega_{0}^{n} =  \int_{X \setminus D} |\mathcal{D}_{\omega_{0}}\phi|^{2} \omega_{0}^{n}.
\]
Since $\overline{\partial} \nabla^{1,0} \phi = \mathcal{D}_{\omega_{0}}\phi = 0$,\
$\nabla^{1,0} \phi$ is a holomorphic vector field on $X \setminus D$.\
By writing locally $\omega_{0} = \sqrt{-1} g_{i,\overline{j}} dz^{i} \wedge d\overline{z}^{j}$, the (1,0)-gradient of $\phi$ can be written as
\[
\nabla^{1,0} \phi = g^{i,\overline{j}} \frac{\partial \phi}{\partial \overline{z}^{j}} \frac{\partial}{\partial z^{i}}.
\]
So, all coefficients $g^{i,\overline{j}} \partial \phi/\partial \overline{z}^{j}$ are holomorphic.\
Moreover, the definition of $\phi$ and the asymptotically conicalness imply differentials of $\phi$ and factors $g^{i,\overline{j}}$ decay near $D$.\
Thus, $\nabla^{1,0} \phi$ can be extended holomorphically to $X$ and vanishes on $D$.\
The hypothesis implies that $\phi$ is constant.\
Since $\phi$ decays near $D$, we have $\phi = 0$ and conclude that $\mathcal{D}_{\omega_{0}}^{*}\mathcal{D}_{\omega_{0}}$ is injective.
\end{proof}

Recall the following fundamental fact (see \cite[Chapter 1, \S 1.4]{Gil}).\

\begin{thm}
\label{fredholm}
Let $L : \mathcal{X} \to \mathcal{Y}$ be a bounded linear operator between Banach spaces $\mathcal{X}$ and $\mathcal{Y}$.\
Then, $L$ is Fredholm if and only if there exists a bounded linear operator $H: \mathcal{Y} \to \mathcal{X}$ such that operators $I_{\mathcal{X}} - H \circ L$ and $I_{\mathcal{Y}} - L \circ H$ are compact.\
Moreover, $H$ is also Fredholm and satisfies
$$
{\rm ind} (L) = - {\rm ind} (H).
$$
\end{thm}
Then, we can show the following.
\begin{lemma}
\label{lemma 2}
If $4 < \delta < 2n $, the operator $\mathcal{D}_{\omega_{0}}^{*}\mathcal{D}_{\omega_{0}}:C^{k,\alpha}_{\delta - 4}(X\setminus D) \to C^{k-4,\alpha}_{\delta}(X\setminus D)$ is a Fredholm operator whose index ${\rm ind}(\mathcal{D}_{\omega_{0}}^{*}\mathcal{D}_{\omega_{0}})$ is zero.\
\end{lemma}

\begin{proof}
Recall the equation
$$
\mathcal{D}^{*}_{\omega_{0}}\mathcal{D}_{\omega_{0}} \phi = \Delta^{2}_{\omega_{0}} \phi + R^{j,\overline{i}} \phi_{i,\overline{j}} + (\nabla^{1,0} \phi, \nabla^{0,1} S(\omega_{0}))_{\omega_{0}}.
$$
Since $(\Delta^{2}_{\omega_{0}})^{-1} : C^{k-4,\alpha}_{\delta}(X\setminus D) \to C^{k,\alpha}_{\delta - 4}(X\setminus D)$ is bounded, it is continuous.\
Consider the linear operator 
$$
I_{C^{k,\alpha}_{\delta - 4}(X\setminus D)} - (\Delta^{2}_{\omega_{0}})^{-1} \circ \mathcal{D}^{*}_{\omega_{0}}\mathcal{D}_{\omega_{0}}.
$$
From the equation above, we obtain
$$
\left( I_{C^{k,\alpha}_{\delta - 4}(X\setminus D)} - (\Delta^{2}_{\omega_{0}})^{-1} \circ \mathcal{D}^{*}_{\omega_{0}}\mathcal{D}_{\omega_{0}} \right)\phi = (\Delta^{2}_{\omega_{0}})^{-1} (R^{j,\overline{i}} \phi_{i,\overline{j}} + (\nabla^{1,0} \phi, \nabla^{0,1} S(\omega_{0}))_{\omega_{0}})
$$
for any $\phi \in C^{k,\alpha}_{\delta - 4}$.\
Since $\phi \in C^{k,\alpha}_{\delta - 4}(X\setminus D)$, the Arzela-Ascoli theorem implies that the operator
$$
\phi \to R^{j,\overline{i}} \phi_{i,\overline{j}} + (\nabla^{1,0} \phi, \nabla^{0,1} S(\omega_{0}))_{\omega_{0}}
$$
is compact.\
The fact that $(\Delta^{2}_{\omega_{0}})^{-1}$ is continuous implies that $I_{C^{k,\alpha}_{\delta - 4}(X\setminus D)} - (\Delta^{2}_{\omega_{0}})^{-1} \circ \mathcal{D}^{*}_{\omega_{0}}\mathcal{D}_{\omega_{0}}$ is also compact.\
Similarly, we obtain the compactness of the operator
$$
I_{C^{k-4,\alpha}_{\delta}(X\setminus D)} - \mathcal{D}^{*}_{\omega_{0}}\mathcal{D}_{\omega_{0}} \circ (\Delta^{2}_{\omega_{0}})^{-1}.
$$
From Theorem \ref{fredholm}, we have finished the proof.
\end{proof}

Then, Lemma \ref{injective} and Lemma \ref{lemma 2} imply that
\begin{prop}
\label{key proposition 1}
If $4 < \delta < 2n$ and there is no nonzero holomorphic vector field on $X$ which vanishes on $D$, the operator $\mathcal{D}_{\omega_{0}}^{*}\mathcal{D}_{\omega_{0}}:C^{k,\alpha}_{\delta - 4}(X\setminus D) \to C^{k-4,\alpha}_{\delta}(X\setminus D)$ is isomorphic and has a bounded inverse.\
\end{prop}

Thus, there exists $K>0$ such that\
\[
|| (\mathcal{D}_{\omega_{0}}^{*}\mathcal{D}_{\omega_{0}})^{-1} ||_{C^{0,\alpha}_{\delta} \to C^{4,\alpha}_{\delta - 4}} < K^{-1}.
\]
In the next section, we study the operator $L_{\omega_{0}}  = - \mathcal{D}_{\omega_{0}}^{*}\mathcal{D}_{\omega_{0}} + (\nabla^{1,0} *, \nabla^{0,1} S(\omega_{0}))_{\omega_{0}}$.\
\begin{rem}
We can show that if the $C_{2}^{1,\alpha}$-norm of $S(\omega_{0})$ is sufficiently small,
there exists the bounded inverse of $L_{\omega_{0}}$ satisfying 
\[
||L_{\omega_{0}}^{-1}||_{C^{0,\alpha}_{\delta} \to C^{4,\alpha}_{\delta-4}} < \hat{K}^{-1}
\]
for some $\hat{K}>0$ (Condition \ref{condition A}).\
\end{rem}

\section{
Proof of Theorem \ref{complete scalar-flat}}
\label{sec:6}

This section also follows from Arezzo-Pacard \cite{AP1}, \cite{AP2} (see also \cite{Sz1}).\
Since we assume that
$$
0 < \hat{S}_{D} < n(n-1),
$$
we can choose a weight $\delta$ so that
\begin{equation}
\label{weight1}
\delta \in (4, \min \{ 2n, 2 + 2n(n-1)/\hat{S}_{D} \})
\end{equation}
Note that if $\theta_{D}$ is cscK, Theorem \ref{scalar curvature decay} implies that $S(\omega_{0}) = O(r^{-\delta})$.\
In addition, from Lemma \ref{geo}, we can choose $\delta$ sufficiently close to $\min \{ 2n, 2 + 2n(n-1)/\hat{S}_{D} \}$ so that a function
\begin{equation}
\label{weight2}
\phi \mathcal{D}_{\omega_{0}}^{*}\mathcal{D}_{\omega_{0}} \phi
\end{equation}
is integrable for $\phi \in C^{4,\alpha}_{\delta - 4}$ with respect to the volume form $\omega_{0}^{n}$.\
Hereafter, we fix a weight $\delta$ satisfying (\ref{weight1}) and (\ref{weight2}).\

\begin{rem}
If $(D,L_{D}) = (\mathbb{P}^{n-1}, \mathscr{O}(1))$, the equality above holds, i.e., $\hat{S}_{D} = n(n-1)$.\
\end{rem}

We will show that the operator $\mathcal{N}:C^{4,\alpha}_{\delta - 4}(X\setminus D) \to C^{4,\alpha}_{\delta-4}(X\setminus D)$ defined in (\ref{operator N}) has a fixed point under Condition \ref{condition A} and Condition \ref{condition B}.\
First, we have

\begin{lemma}
\label{contraction lemma}
There exists $c_{0} > 0$ depending only on $\omega_{0}$ such that if $||\phi||_{C^{4,\alpha}_{-2}(X\setminus D)} \leq c_{0}$, we have
$$
||L_{\omega_{\phi}} - L_{\omega_{0}}||_{C^{4,\alpha}_{\delta - 4} \to C^{0,\alpha}_{\delta}} \leq \hat{K}/2
$$
and $\omega_{\phi} = \omega_{0} + \dol \phi$ is positive.\
\end{lemma}

\begin{proof}
Note that
\begin{equation}
\label{inv1}
g_{\phi}^{-1} - g^{-1} = g_{\phi}^{-1} ( g - g_{\phi} ) g^{-1}
\end{equation}
for $\phi$ such that $\omega_{\phi} = \omega_{0} + \dol \phi$ is positive.\

Locally, we can write as
\begin{equation}
\label{Ls}
L_{\omega_{\phi}} \psi = - \Delta_{\omega_{\phi}}^{2} \psi - R_{\omega_{\phi}}^{i,\overline{j}} \psi_{i,\overline{j}}.
\end{equation}
For instance, we have\\

$||(r^{2} + 1)^{\delta/2} (g_{\phi}^{i,\overline{j}}g_{\phi}^{k,\overline{l}} - g^{i,\overline{j}}g^{k,\overline{l}}) \psi_{i,\overline{j},k,\overline{l}}||_{C^{0,\alpha}}$
\begin{eqnarray*}
&\leq& ||(r^{2} + 1)^{4/2} (g_{\phi}^{i,\overline{j}}g_{\phi}^{k,\overline{l}} - g^{i,\overline{j}}g^{k,\overline{l}})||_{C^{0,\alpha}} ||\psi||_{C^{4,\alpha}_{\delta - 4}(X\setminus D)}\\
&=& ||(r^{2} + 1)^{4/2} ( g_{\phi}^{i,\overline{j}}(g_{\phi}^{k,\overline{l}} - g^{k,\overline{l}}) + (g_{\phi}^{i,\overline{j}} - g^{i,\overline{j}})g^{k,\overline{l}})||_{C^{0,\alpha}} ||\psi||_{C^{4,\alpha}_{\delta - 4}(X\setminus D)}.\\
\end{eqnarray*}
Note that $g_{\phi}^{i,\overline{j}} = O(r^{-2})$ as $r \to \infty$, if $c_{0}$ is sufficiently small.\
Thus, if $c_{0}$ is sufficiently small, the equation (\ref{inv1}) implies that the term above can be made small arbitrarily.\
Applying the same argument to remainders in (\ref{Ls}), we can use the asymptotically conicalness to obtain the desired result.
\end{proof}

To show that the operator $\mathcal{N} : C^{4,\alpha}_{\delta - 4} \to C^{4,\alpha}_{\delta - 4}$ is a contraction, we need the following lemma :

\begin{lemma}
\label{mean value}
Assume that
$$
||\phi||_{C^{4,\alpha}_{\delta - 4}},\hspace{5pt} ||\psi||_{C^{4,\alpha}_{\delta - 4}} \leq c_{0}.
$$
Then, we have
$$
||\mathcal{N}(\phi) - \mathcal{N}(\psi)||_{C^{4,\alpha}_{\delta - 4}} \leq \frac{1}{2} ||\phi - \psi||_{C^{4,\alpha}_{\delta - 4}}.
$$
\end{lemma}

\begin{proof}
Since the operator $\mathcal{N}$ is defined by $\mathcal{N}(\phi) := - L_{\omega_{0}}^{-1}( S(\omega_{0}) + Q_{\omega_{0}}(\phi) )$, we have
$$
\mathcal{N}(\phi) - \mathcal{N}(\psi) = - L_{\omega_{0}}^{-1}( Q_{\omega_{0}}(\phi) - Q_{\omega_{0}}(\psi) ).
$$
The mean value theorem implies that there exists $\chi = t \phi + (1 - t) \psi$ for $t \in [0,1]$ such that
$$
DQ_{\omega_{0},\chi} (\phi - \psi) = Q_{\omega_{0}}(\phi) - Q_{\omega_{0}}(\psi),
$$
and the direct computation implies that
$$
DQ_{\omega_{0},\chi} = L_{\omega_{\chi}} - L_{\omega_{0}}.
$$
We know that  $||\phi||_{C^{4,\alpha}_{-2}} \leq||\phi||_{C^{4,\alpha}_{\delta -4}} \leq c_{0}$.\
Using Lemma \ref{contraction lemma}, we finish the proof.
\end{proof}

The following Proposition implies that the existence of a complete scalar-flat {\kah} metric.\

\begin{prop}
\label{key proposition 2}
Set
$$
\mathcal{U} := \left\{ \phi \in C^{4,\alpha}_{\delta - 4} : ||\phi||_{C^{4,\alpha}_{\delta -4}} \leq c_{0} \right\}.
$$
If Condition \ref{condition A} and Condition \ref{condition B} hold, the operator $\mathcal{N}$ is a contraction on $\mathcal{U}$ and $\mathcal{N}(\mathcal{U}) \subset \mathcal{U}$.\
\end{prop}

\begin{proof}
By the condition $||\phi||_{C^{4,\alpha}_{\delta -4}} \leq c_{0}$, Lemma \ref{contraction lemma} and Lemma \ref{mean value} obviously imply that $||\phi||_{C^{4,\alpha}_{-2}} \leq c_{0}$, $\mathcal{N}$ is a contraction on $\mathcal{U}$.\
Immediately, we have
$$
||\mathcal{N}(\phi)||_{C^{4,\alpha}_{\delta - 4}} \leq ||\mathcal{N}(\phi) - \mathcal{N}(0) ||_{C^{4,\alpha}_{\delta - 4}}  + || \mathcal{N}(0) ||_{C^{4,\alpha}_{\delta - 4}}.
$$
From the previous lemma, we obtain:\
$$
||\mathcal{N}(\phi) - \mathcal{N}(0) ||_{C^{4,\alpha}_{\delta - 4}} \leq \frac{1}{2} c_{0}.
$$
The hypothesis in this proposition implies that
\begin{eqnarray*}
|| \mathcal{N}(0) ||_{C^{4,\alpha}_{\delta - 4}}
&\leq& \hat{K}^{-1} || S(\omega_{0}) ||_{C^{0,\alpha}_{\delta}} \leq \frac{1}{2} c_{0}.
\end{eqnarray*}
Thus, $\mathcal{N}(\phi) \in \mathcal{U}$.
\end{proof}

{\it Proof of Theorem \ref{complete scalar-flat}.}
If Condition \ref{condition A} and Condition \ref{condition B} hold, Proposition \ref{key proposition 2} implies that there is a unique $\phi_{\infty} := \lim_{i \to \infty} \mathcal{N}^{i}(\phi)$ for any $\phi \in \mathcal{U} \subset C^{4,\alpha}_{\delta - 4}$ satisfying $\phi_{\infty} = \mathcal{N}(\phi_{\infty})$.\
Therefore, $\omega_{0} + \dol \phi_{\infty}$ is a complete scalar-flat {\kah} metric on $X \setminus D$.

\sq


%
%



\bigskip
\address{
(CURRENT ADDRESS)\\
Osaka Prefectural Abuno High School,\\
3-38-1, Himuro-chou, Takatsuki-shi,\\
Osaka, 569-1141\\
Japan
}
{takahiro.aoi.math@gmail.com}

\bigskip
\address{
(OLD ADDRESS)\\
Department of Mathematics\\
Graduate School of Science\\
Osaka University\\
Toyonaka 560-0043\\
Japan
}
{t-aoi@cr.math.sci.osaka-u.ac.jp}

\end{document}